\documentclass[10pt]{article}
\usepackage{amsmath}
\usepackage{amsfonts}
\usepackage{amsthm}
\usepackage{graphicx}
\usepackage{amssymb}
\newtheorem{theorem}{Theorem}

\newtheorem{example}{Example}

\newenvironment{remark}[1][Remark]{\begin{trivlist}
\item[\hskip \labelsep {\bfseries #1}]}{\end{trivlist}}

\DeclareMathOperator{\hess}{Hess}
\DeclareMathOperator{\supp}{supp}
\providecommand{\abs}[1]{\lvert#1\rvert}
\providecommand{\inner}[2]{\langle#1 , #2\rangle}
\providecommand{\supover}[3]{ {\underset{#1 \in #2}{\sup}} { #3(#1) }}
\providecommand{\maxover}[3]{ {\underset{#1 \in #2}{\max}} { #3(#1) }}
\numberwithin{equation}{section}

\title{A Liouville-type Theorem for Smooth Metric Measure Spaces}
\author{Kevin Brighton\footnote{Department of Mathematics, University of California, Santa Barbara, CA 93106, USA. E-mail: kcbrighton@math.ucsb.edu}}
\begin{document}

\maketitle

\begin{abstract}
For smooth metric measure spaces $(M, g, e^{-f} dvol)$ we prove a Liuoville-type theorem when the Bakry-Emery Ricci tensor is nonnegative. ÊThis generalizes a result of Yau, which is recovered in the case $f$ is constant. This result follows from a gradient estimate for f-harmonic functions on smooth metric measure spaces with Bakry-Emery Ricci tensor bounded from below.
\end{abstract}

\section{Introduction}

 In this paper we study smooth metric measure spaces.  A smooth metric measure space is a complete Riemannian manifold together with a measure conformal to the Riemannian measure. More precisely a smooth metric measure space is a triple, $(M^n , g , e^{-f} dvol)$, where M is a complete Riemannian manifold of dimension $n$, $f$ is a smooth real-valued function on $M$ and $dvol$ is the Riemannian volume density.  

Smooth metric measure spaces carry natural analogs of the Ricci tensor.  In this paper we study the Bakry-Emery Ricci tensor, or $\infty$-Bakry-Emery Ricci tensor, defined by $Ric_f = Ric + \hess (f)$.  This tensor can be thought of as the limit of the N-Bakry-Emery Ricci tensors defined by $Ric^N_f = Ric + \hess(f)- \frac{1}{N} df \otimes df$.  The Bakry-Emery Ricci tensor, as well as its finite dimensional variants, were studied extensively by Bakry and Emery in their investigations of diffusion processes.  It also arises in the theory of the Ricci flow, as $Ric_f = \lambda g$ is precisely the gradient Ricci soliton equation.  Smooth metric measure spaces also carry a natural analog of the Laplace-Beltrami operator, the f-Laplacian, defined for a function $u$ by $\Delta_f (u)= \Delta(u) - \inner{\nabla u}{\nabla f}$.

 A great deal is known about the geometric and topological implications of the Ricci tensor.   The natural question, first studied by Lichnerowicz \cite{MR0268812,MR0300228}, is which of these results hold for the Bakry-Emery Ricci tensor.  This is an active field of current research. See, for example \cite{MR2170766,MR2016700,arXiv:math.DG/0612532,MR2577473, arXiv:1003.0191} or \cite{weisurvey} and the references therein. 
 
In $1976$ Yau  \cite{MR0417452} famously showed that any positive or bounded harmonic function on a complete Riemannian manifold with $Ric \geq 0$ must be constant.  The goal of this paper is to establish a corresponding result for f-harmonic functions on smooth metric measure spaces with $Ric_f \geq 0$.  For positive f-harmonic functions this was shown in \cite{arXiv:math.DG/0612532} for smooth metric measure spaces with uniformly positive Bakry-Emery Ricci tensor.  Such a result, however,  is impossible for positive f-harmonic functions on smooth metric measure spaces with nonnegative Bakry-Emery Ricci tensor, as the following example illustrates:
\begin{example}
\label{example}
Consider $\mathbb{R}^n$ with the usual metric and let $f(x) = x_1$, the first coordinate function.  Clearly $Ric_f \geq 0$ in this case.  Now consider the function $u(x) = e^{x_1}$ and note
\begin{equation}
\Delta_f (u) = \Delta(u) - \inner{\nabla u}{ \nabla f} = u'' - u' =0.
\end{equation}
Thus $u$ is a non-constant positive f-harmonic function on $(\mathbb{R}^n, g , e^{-f} dvol)$.
\end{example}
On the other hand for smooth metric measure spaces with nonnegative Bakry-Emery Ricci tensor we show that bounded f-harmonic functions are constant.  
\begin{theorem}
\label{Liouville}
Let $(M,g,{e^{-f}} dvol)$ be a complete smooth metric measure space with ${Ric_f \geq 0}$.  If $u$ is a bounded f-harmonic function defined on $M$ then $u$ is constant.
\end{theorem}
In the case of the N-Bakry-Emery Ricci tensor, for finite N, similar Liouville-type theorems are obtained in \cite{MR2170766, MR2487589, Hua2010, MR2296933}.  The Liouville-type theorem is proved by establishing the following estimate on the gradient of an f-harmonic function.
  \begin {theorem}
  \label{uestimate}
Let $(M,g,{e^{-f}} dvol)$ be a complete smooth metric measure space with ${Ric_f \geq -(n-1)H^2}$ where $H \geq 0$.  If $u$ is a positive f-harmonic function defined on $\overline{B(q;2R)}$ with $R \geq 1$, then for any $q_0 \in \overline{B(q;R)}$,
\begin{equation}
\abs{\nabla u}(q_0) \leq \sqrt{\frac{c_1(n, \alpha)}{R} + c_2(n)H^2} \supover{p}{B(q;2R)}{u},
\end{equation}
\noindent where $\alpha = \maxover{p}{\{p:d(p,q)=1\}}{\Delta_f r}$.
\end{theorem}
\noindent The constants above are given explicitly in \eqref{final}.  In the case $H=0$, Theorem~\ref{Liouville} then immediately follows by taking a limit as $R$ tends to infinity.  The next section is dedicated to the proof of this estimate.  A related estimate is obtained in \cite{arXiv:0902.2226}.  

The idea of Yau's original proof was to apply the Bochner formula to the log of a positive harmonic function to obtain an initial estimate, then multiply by a cut-off function and apply the maximum principle and Laplacian comparison.  On smooth metric measure spaces there is a Bochner formula which looks similar to the classical Bochner formula, see \cite{MR2577473} for example.  This formula is given by
\begin{equation}
\frac{1}{2} \Delta_f \abs{\nabla u}^2 = \abs{ \hess \, u}^2 + \inner{\nabla u}{\nabla( \Delta_f u)} + Ric_f (\nabla u , \nabla u).
\end{equation}
The difference between the two Bochner formulae is the Hessian term which is controlled by the Laplacian.  In the context of smooth metric measure spaces the Laplacian is related to the f-Laplacian by $\Delta u = \Delta_f u +\inner{\nabla u}{\nabla f}$.  Thus, when squared, the first and last terms will be positive, so the issue is controlling the cross term.  The cross term is controlled by choosing a suitable function $h$ of the given f-harmonic function and examining two cases depending on $h$ and the relative magnitudes of $\Delta h$ and  $\inner{\nabla h}{\nabla f}$. The cross term is then collected with the larger of the leading and final terms to obtain an initial estimate.  Then we multiply by a cut-off function, use the maximum principle and apply the f-Laplacian estimate of Wei-Wylie \cite[Theorem 2.1]{MR2577473} to control the resulting terms.

In the final section we use the the ideas of the proof of Theorem 2 to obtain an estimate for $\abs{\nabla \log(u)}$ as in Yau \cite{MR0417452}.  However, choosing $h= \log(u)$ a gap develops between the constants required to define the two cases.  Thus the estimate is only established with an additional assumption on the relative sizes of $\nabla \log(u)$ and $\inner{\nabla \log(u)}{\nabla f}$.  Specifically we establish the following:
\begin {theorem}
\label{logestimate}
Let $(M,g,{e^{-f}} dvol)$ be a smooth complete metric measure space with $Ric_f \geq -(n-1)H^2$.  If $u$ is a positive f-harmonic function defined on $\overline{B(q;2R)}$ with $R \geq 1$, such that either $\inner{\nabla f}{\nabla \log(u)} \leq a\abs{\nabla \log (u)}^2$ for some $0\leq a<\frac{1}{2}$ or $\inner{\nabla f}{\nabla \log(u)} \geq 2\abs{\nabla \log (u)}^2$ then for any $q_0 \in \overline{B(q;R)}$,
\begin{equation}
 \abs{\nabla (\log u)}(q_0)  \leq \sqrt{c_1(n, a)H^2 + \frac{c_2(n,a,\alpha)}{R}}.
\end{equation}
\end{theorem}
\begin{remark}
Note that when $f$ is constant this recovers the original estimate of Yau \cite{MR0417452}.
\end{remark}
\begin{remark}
	The assumption on the relative sizes of $\inner{\nabla f}{\nabla \log(u)}$ and $\abs{\nabla \log (u)}^2$ here is required as Example \ref{example} demonstrates.  Note that in Example \ref{example}, $\frac{\inner{\nabla f}{\nabla \log(u)}}{\abs{\nabla \log (u)}} = 1$.
\end{remark}

\section{Proof of Theorems 1 and 2}
The proof of Theorem~\ref{uestimate} will follow from applying the Bochner formula to an appropriate function $h$ of a given positive f-harmonic function $u$.  As noted above, the problem term in the Bochner formula is the Hessian term, which is controlled with the Laplacian.  Indeed, if $\inner{\nabla u}{\nabla f}$ is small in relation to $\Delta u$, then $\Delta u$ can be controlled by $\Delta_f u$ since $\Delta u = \Delta_f u +\inner{\nabla u}{\nabla f}$.   Applying the Bochner formula to $\log(u)$ then yields the desired estimate, just as in the classical case.  On the other hand, if  $\inner{\nabla u}{\nabla f}$ is large in relation to $\Delta u$ this approach fails.  In this case $\Delta u$ is controlled by $\inner{\nabla u}{\nabla f}$.  Here, applying the Bochner formula directly to $u$ will give the estimate.  These two estimates can be combined directly, but a more convenient method suggested by Dr. Zhiqin Lu is to apply the Bochner formula to $u^\epsilon$ for some $\epsilon \in (0,1)$.  This is the proof we follow below.
\begin{proof}[Proof of Theorem~\ref{uestimate}]
\noindent Let $h=u^{\epsilon}$, where $\epsilon \in (0,1)$.  Applying the Bochner formula to $h$, we find:
\begin{equation} \label{bochner}
\frac{1}{2} \Delta_f \abs{\nabla h}^2 = \abs{ \hess \, h}^2 + \inner{\nabla h}{\nabla( \Delta_f h)} + Ric_f (\nabla h , \nabla h).
\end{equation}
Now we compute,
\begin{equation} 
\begin{split}
\Delta_f h &= \Delta h - \inner{\nabla h}{\nabla f},\\
                  &= \epsilon (\epsilon -1) u^{\epsilon-2} \abs{\nabla{u}}^2 +\epsilon u^{\epsilon-1}\Delta_f u, \\
                  &=\frac{(\epsilon-1)\abs{\nabla{h}}^2}{\epsilon h}.
\end{split}
\end{equation}
 Where the last equality uses the fact that $u$ is f-harmonic.  Further $Ric_f (\nabla h , \nabla h)  \geq -(n-1)H^2 \abs{\nabla h}^2$ and by the Schwarz inequality we find:
\begin{equation} \label{hessh}
\begin{split}
\abs{\hess \, h}^2 &\geq \frac{\abs{\Delta h}^2}{n},\\
                             &= \frac{1}{n}(\Delta_f h + \inner{\nabla f}{\nabla h})^2,\\
                             &=\frac{1}{n}(\frac{(\epsilon-1)\abs{\nabla{h}}^2}{\epsilon h} + \inner{\nabla f}{\nabla h})^2.\\
 \end{split}
 \end{equation}
 Also,
\begin{equation} \label{mixed}
\begin{split}
\inner{\nabla h}{\nabla (\Delta_f h)}&=\inner{\nabla h}{\nabla (\frac{(\epsilon-1)\abs{\nabla{h}}^2}{\epsilon h})},\\
                                                              &=\frac{\epsilon-1}{\epsilon h}\inner{\nabla h}{\nabla \abs{\nabla h}^2} - \frac{(\epsilon-1)\abs{\nabla h}^4}{\epsilon h^2}.
\end{split}
\end{equation}
Thus substituting \eqref{hessh} and \eqref{mixed} into \eqref{bochner},
\begin{equation}
\begin{split}
\frac{1}{2} \Delta_f \abs{\nabla h}^2 &\geq \frac{ (\epsilon-1)^2\abs{\nabla h}^4}{n\epsilon^2 h^2} + \frac{2(\epsilon-1) \abs{ \nabla h}^2}{n\epsilon h} \inner{\nabla h}{\nabla f} + \frac{1}{n}\inner{\nabla h}{\nabla f}^2\\
                                                               &\qquad+ \frac{\epsilon-1}{\epsilon h}\inner{\nabla h}{\nabla \abs{\nabla h}^2} - \frac{(\epsilon-1)\abs{\nabla h}^4}{\epsilon h^2} -(n-1)H^2 \abs{\nabla h}^2.
\end{split}
\end{equation}

We now examine two cases according to which term in $\Delta h$ dominates.  First, suppose $p \in \overline{B(q;2R)}$ such that $\inner{\nabla h}{\nabla f} \leq a\frac{\abs{\nabla h}^2}{h}$, for some constant $a$ to be determined.  Thus at $p$ we find:

\begin{equation}
\begin{split}
\frac{1}{2} \Delta_f \abs{\nabla h}^2 &\geq \frac{ (\epsilon-1)^2\abs{\nabla h}^4}{n\epsilon^2 h^2} + \frac{2(\epsilon-1) \abs{ \nabla h}^2}{n\epsilon h} (\frac{a\abs{\nabla h}^2}{h}) + \frac{1}{n}\inner{\nabla h}{\nabla f}^2\\
                                                                 & \qquad + \frac{\epsilon-1}{\epsilon h}\inner{\nabla h}{\nabla \abs{\nabla h}^2} - \frac{(\epsilon-1)\abs{\nabla h}^4}{\epsilon h^2} -(n-1)H^2 \abs{\nabla h}^2,\\
&\geq \frac{(\epsilon-1)^2 +2a\epsilon (\epsilon-1) -n\epsilon (\epsilon-1)}{n\epsilon^2}\frac{\abs{\nabla h}^4}{h^2} + \frac{\epsilon-1}{\epsilon h}\inner{\nabla h}{\nabla \abs{\nabla h}^2} -(n-1)H^2 \abs{\nabla h}^2.
\end{split}
\end{equation}
We note that the coefficient of $\frac{\abs{\nabla h}^4}{h^2}$ will be positive whenever
\begin{equation}
a < \frac{n\epsilon - \epsilon +1}{2\epsilon}.
\end{equation}
Now suppose to the contrary that $p \in \overline{B(q;2R)}$ such that $\inner{\nabla h}{\nabla f} \geq a\frac{\abs{\nabla h}^2}{h}$.  Thus at $p$ we find
\begin{equation}
\begin{split}
\frac{1}{2} \Delta_f \abs{\nabla h}^2 &\geq \frac{ (\epsilon-1)^2\abs{\nabla h}^4}{n\epsilon^2 h^2} + \frac{2(\epsilon-1)}{n\epsilon h}(\frac{h}{a}\inner{\nabla h}{\nabla f}) \inner{\nabla h}{\nabla f} + \frac{1}{n}\inner{\nabla h}{\nabla f}^2\\
                                                                & \qquad + \frac{\epsilon-1}{\epsilon h}\inner{\nabla h}{\nabla \abs{\nabla h}^2} - \frac{(\epsilon-1)\abs{\nabla h}^4}{\epsilon h^2} -(n-1)H^2 \abs{\nabla h}^2,\\
                                                                &=\frac{(\epsilon-1)^2 - n(\epsilon-1)}{n\epsilon^2 h^2}\abs{\nabla h}^4 + \frac{2(\epsilon-1) +\epsilon a}{n\epsilon a}\inner{\nabla h}{\nabla f}^2 + \frac{\epsilon-1}{\epsilon h}\inner{\nabla h}{\nabla \abs{\nabla h}^2} -(n-1)H^2 \abs{\nabla h}^2.
\end{split}
\end{equation}
Thus the coefficient of $\inner{\nabla h}{\nabla f}^2 $ will be positive whenever
\begin{equation}
a > \frac{2(1-\epsilon)}{\epsilon}.
\end{equation}
Hence provided $\epsilon \in (\frac{3}{4} ,1)$ there will exist an $a$ satisfying both of these inequalities.  In particular, choosing $\epsilon = \frac{7}{8}$, we may take $a= \frac{1}{2}$ and in either case,
\begin{equation} \label{estimate}
\frac{1}{2} \Delta_f \abs{\nabla h}^2 \geq (\frac{7n-6}{49nh^2}) \abs{\nabla h}^4 - \frac{1}{7h}\inner{\nabla h}{\nabla \abs{\nabla h}^2} -(n-1)H^2 \abs{\nabla h}^2.
\end{equation}
Therefore this estimate holds on all of $\overline{B(q;2R)}$. 

Now we wish to multiply $\abs{\nabla h}^2$ by a cut-off function to obtain a function which achieves its maximum inside $B(q;2R)$.  Let $g:[0,2R] \rightarrow [0,1]$ be a $C^3$ function with the following properties: \\
a.)$ g \vert _{[0,R]} =1$,\\
b.) $\supp(g) \subseteq [0,2R)$, \\
c.) $ \frac{-c}{R}\sqrt{g} \leq g\prime \leq 0$,\\
d.) $\abs{g \prime \prime} \leq \frac{c}{R^2}$,\\
for some $c > 0$.  Define $\phi : \overline{B(q;2R)} \rightarrow [0,1]$ by $\phi(x)=g(d(x,q))$ and set $G=\phi \abs{\nabla h}^2$.  We first compute $\nabla G = (\nabla \phi)(\abs{\nabla h}^2) + (\phi)(\nabla \abs{\nabla h}^2)$.  Hence
\begin{equation} \label{gradsquared}
\abs{\nabla h}^2 = \frac{G}{\phi},
\end{equation}
and
\begin{equation} \label{gradgradsquared}
\nabla \abs{\nabla h}^2 = \frac{\nabla G}{\phi} - \frac{\nabla \phi}{\phi ^2}G.
\end{equation}
Further, 
\begin{equation}
\begin{split}
\Delta_f G &=(\Delta \phi)(\abs{\nabla h}^2) + 2\inner{\nabla \phi}{ \nabla \abs{\nabla h}^2} + (\phi)(\Delta \abs{\nabla h}^2) - \inner{\nabla f}{ (\nabla \phi)\abs{\nabla h}^2} - \inner{\nabla f}{ (\phi)(\nabla \abs{\nabla h}^2)},\\
&= (\Delta \phi)(\abs{\nabla h}^2) + 2\inner{\nabla \phi}{ \nabla \abs{\nabla h}^2}  - \inner{\nabla f}{ (\nabla \phi)(\abs{\nabla h}^2)} + \phi \Delta_f \abs{\nabla h}^2.
\end{split}
\end{equation}
Thus, 
\begin{equation} \label{flaplace}
\begin{split}
\Delta_f \abs{\nabla h}^2 &= \frac{1}{\phi} \Delta_f G - \frac{\abs{\nabla h}^2}{\phi} \Delta_f \phi - 2\inner{\frac{\nabla \phi}{\phi}}{ \nabla\abs{\nabla h}^2},\\
&=\frac{1}{\phi} \Delta_f G- \frac{G}{\phi^2} \Delta_f \phi - 2\inner{\frac{\nabla \phi}{\phi}}{ \frac{\nabla G}{\phi} - \frac{\nabla \phi}{\phi^2}G}.
\end{split}
\end{equation}
Now substituting \eqref{gradsquared}, \eqref{gradgradsquared} and \eqref{flaplace} into \eqref{estimate} and solving for $\frac{1}{\phi} \Delta_f G$ we find that
\begin{equation}
\frac{1}{\phi} \Delta_f G \geq \frac{G}{\phi^2} \Delta_f \phi +2\inner{\frac{\nabla \phi}{\phi}}{ \frac{\nabla G}{\phi}} - 2\inner{\frac{\nabla \phi}{\phi}}{\frac{\nabla \phi}{\phi^2}G}  + (\frac{14n-12}{49nh^2}) \frac{G^2}{\phi^2} - \frac{2}{7h}\inner{\nabla h}{\frac{\nabla G}{\phi} - \frac{\nabla \phi}{\phi ^2}G} -2(n-1)H^2\frac{G}{\phi}.
\end{equation}

Now $G$ must achieve its maximum on the compact set $\overline{B(q;2R)}$ and since $G$ vanishes on the boundary, it must attain its maximum at some point $q_0 \in B(q;2R)$.  We assume $q_0$ is not in the cut locus of $q$, if it is we adjust $\phi$ slightly and interpret the following inequalities in the barrier sense.  Evaluating the above at $q_0$ and solving for the $G$ term,

\begin{equation}\label{almostdone}
(\frac{14n-12}{49nh^2}) G \leq -\Delta_f \phi+\frac{2}{\phi} \abs{\nabla \phi}^2-\frac{2}{7h}\inner{\nabla h}{\nabla \phi} +2(n-1)\phi H^2.
\end{equation}

Now we note that if $q_0 \in B(q;1)$, then since $\phi$ is the identity on $\overline{B(q;R)}$, $q_0$ is the maximum point of $\abs{\nabla h}$ on $\overline{B(q;R)}$.  Thus \eqref{almostdone} yields, 

\begin{equation}
\abs{\nabla u} \leq \sqrt{\frac{64n(n-1)}{7n-6}} H \supover{p}{B(q;2R)}{u},
\end{equation}
on ${B(q;R)}$, which is the desired estimate.  Thus we may assume $q_0 \notin B(q;1)$.
Now to control the f-Laplacian term in \eqref{almostdone}, note that $\phi$ is a radial function and hence $\Delta_f \phi = g^\prime \Delta_f r + g^{\prime \prime}$, where $r(x)=d(x,q)$.  Now since ${Ric_f \geq -(n-1)H^2}$ and $d(q,q_0) \geq 1$, the Laplacian estimate of \cite{MR2577473} gives,
\begin{equation}
\Delta_f r(q_0) \leq \alpha + (n-1)H^2 (2R-1).
\end{equation}
Further by the construction of $\phi$, $\frac{-c}{R} \sqrt{g} \leq g^\prime \leq 0$ and $g^{\prime \prime} \geq -\frac{c}{R^2}$.  Hence,
\begin{equation} \label{deltaphibd}
\begin{split}
\Delta_f \phi &\geq \frac{-c}{R}\Delta_f(r) -\frac{c}{R^2},\\ 
                       & \geq \frac{-c \alpha}{R} - c(n-1)H^2 \frac{2R-1}{R} -\frac{c}{R^2}.
\end{split}
\end{equation}
Finally to control the $\inner{\nabla h}{ \nabla \phi}$ term, note that for any $\delta \geq 0$,
\begin{equation} \label{crossterm}
2\abs{\inner{\nabla h}{ \nabla \phi}} \leq 2\abs{\nabla h} \abs{\nabla \phi}  \leq \frac{\delta \phi}{h} \abs{\nabla h}^2 + \frac{h}{\delta \phi} \abs{ \nabla \phi}^2.
\end{equation}

Thus, substituting \eqref{crossterm} and \eqref{deltaphibd} into \eqref{almostdone} and collecting terms we find:
\begin{equation} 
(\frac{14n-12}{49nh^2} - \frac{\delta}{7h^2}) G \leq \frac{c \alpha}{R} +\frac{c}{R^2}+(2+\frac{1}{7\delta})\frac{\abs{\nabla \phi}^2}{\phi} +2(n-1)H^2(\phi + c \frac{R-1}{R}) .
\end{equation}
Now choosing $\delta=\frac{1}{7}$, noting that $\frac{\abs{\nabla \phi}^2}{\phi} \leq \frac{c^2}{R^2}$ and using $R \geq 1$ we find,
\begin{equation} 
\frac{13n-12}{49nh^2}G \leq \frac{c \alpha + c+3c^2}{R} +2(n-1)H^2(\phi + c) .
\end{equation}
Finally restricting to $B(q;R)$ and using the definition of $h$ we conclude that
\begin{equation} \label{final}
\abs{\nabla u} \leq \sqrt{ \frac{64n(c \alpha + c+3c^2)}{13n-12} \frac{1}{R} +\frac{128n(n-1)(1+ c)}{13n-12}H^2} \supover{p}{B(q;2R)}{u}.
\end{equation}
\end{proof}

\section{Proof of Theorem 3}

In this section we prove Theorem~\ref{logestimate}.  As noted above the proof begins by applying the Bochner formula to $\log(u)$ and then using the assumption to control the Hessian term with $\Delta_f \log(u)$.

\begin{proof}[Proof of Theorem~\ref{logestimate}]
Define $h=\log(u)$. First we note that,
\begin{equation}
\begin{split}
\Delta_f h &=\Delta h - \inner{\nabla f}{ \nabla h}, \\
                  &=\frac 1{u} \Delta_f u - \abs{ \frac {\nabla u} {u} }^2.
\end{split}
\end{equation}
Thus since $u$ is f-harmonic, $\Delta_f h = - \abs{ \nabla h}^2$.  Now applying the Bochner formula we find:
\begin{equation}
\frac 1{2} \Delta_f \abs{ \nabla h}^2 = \abs{\hess \, h}^2 + \inner{\nabla h}{ \nabla (\Delta_f h)} + Ric_f (\nabla h, \nabla h).
\end{equation}
Further, by the Schwarz inequality $\abs{\hess \, h}^2 \geq \frac {\abs{\Delta h}^2}{n}$. Now,
\begin{equation}
\begin{split}
\abs{ \Delta h}^2 &= \abs{ \Delta_f h + \inner{\nabla h}{\nabla f}}^2,\\
                               &=\abs{ - \abs{ \nabla h}^2 + \inner{\nabla h}{\nabla f}}^2,\\
                               &= \abs{ \nabla h}^4 - 2\abs{\nabla h}^2 \inner{\nabla h}{\nabla f} + \inner{\nabla h}{\nabla f}^2.
 \end{split}
 \end{equation}
Thus if the first assumption holds we find:
\begin{equation}
\begin{split}
\abs{ \Delta h}^2 &\leq \abs{ \nabla h}^4 - 2a\abs{\nabla h}^4,\\
                               &=(1- 2a)\abs{ \nabla h}^4.
 \end{split}
 \end{equation}
If the second assumption holds, then:
 \begin{equation}
\begin{split}
\abs{ \Delta h}^2&\leq \abs{ \nabla h}^4 - \inner{\nabla f}{\nabla h}^2 + \inner{\nabla h}{\nabla f}^2,\\
                              &\leq \abs{ \nabla h}^4.
 \end{split}
 \end{equation}
Also, $Ric_f (\nabla h, \nabla h) \geq -(n-1)H^2 \abs{\nabla h}^2$. Thus substituting the above into the Bochner formula in either case we find that,
 \begin{equation} \label{sboch2}
 \frac{1}{2} \Delta_f \abs{\nabla h}^2 \geq \frac{1-2a}{n} \abs{\nabla h}^4 - \inner{\nabla h}{\nabla(\abs{\nabla h}^2)}-(n-1)H^2 \abs{\nabla h}^2.
 \end{equation}
From here the proof follows the same argument as Theorem~\ref{uestimate}, yielding
\begin{equation}
\abs{\nabla h} \leq \sqrt{\left({\frac{2nc\alpha + 2nc}{3(1-2\alpha)} +\frac{2n(2n+2-4a)c^2}{3(1-2\alpha)^2}}\right) \frac{1}{R} + \frac{4n(c+1)(n-1)}{3(1-2a)}H^2},
\end{equation}
where $a$ is taken to be $0$ in the case the second assumption holds.
\end{proof}

\section{Acknowledgements}
I wish to thank Guofang Wei for her guidance, tutelage and support.  In addition, I'd like to thank Zhiqin Lu and Jeffrey Case for their valuable discussions and advice. 

\bibliographystyle{plain}
 \bibliography{citations}
\end{document}